\documentclass[11pt]{article}
\usepackage{amsmath, amssymb, amsthm, verbatim,enumerate,bbm}
\usepackage{indentfirst}

\title{Rainbow Hamilton cycles in random graphs and hypergraphs}

\author{ Asaf Ferber
\thanks{Department of Mathematics, Yale University, and Department of Mathematics, MIT. Emails:
asaf.ferber@yale.edu, and ferbera@mit.edu.}\and Michael Krivelevich
\thanks{School of Mathematical Sciences, Raymond and Beverly
Sackler Faculty of Exact Sciences, Tel Aviv University, Tel Aviv,
6997801, Israel. Email: krivelev@post.tau.ac.il. Research supported in
part by USA-Israel BSF grant 2010115 and by grant 912/12 from the Israel
Science Foundation.}}

\date{\today}
\allowdisplaybreaks

\theoremstyle{plain}
\newtheorem{theorem}{Theorem}[section]
\newtheorem{lemma}[theorem]{Lemma}

\newtheorem{remark}[theorem]{Remark}
\newtheorem{definition}[theorem]{Definition}

\newcommand{\Bin}{\ensuremath{\textrm{Bin}}}

\begin{document}
\maketitle

\begin{abstract}
Let $H$ be an edge colored hypergraph. We say that $H$ contains a
\emph{rainbow} copy of a hypergraph $S$ if it contains an isomorphic
copy of $S$ with all edges of distinct colors.

We consider the following setting. A randomly edge colored random
hypergraph $H\sim \mathcal H_c^k(n,p)$ is obtained by adding each
$k$-subset of $[n]$ with probability $p$, and assigning it a color
from $[c]$ uniformly, independently at random.

As a first result we show that a typical $H\sim \mathcal H^2_c(n,p)$
(that is, a random edge colored graph) contains a rainbow Hamilton
cycle, provided that $c=(1+o(1))n$ and $p=\frac{\log n+\log\log
n+\omega(1)}{n}$. This is asymptotically best possible with respect
to both parameters, and improves a result of Frieze and Loh.

Secondly, based on an ingenious coupling idea of McDiarmid, we
provide a general tool for tackling problems related to finding
``nicely edge colored" structures in random graphs/hypergraphs. We
illustrate the generality of this statement by presenting two
interesting applications. In one application we show that a typical
$H\sim \mathcal H^k_c(n,p)$ contains a rainbow copy of a hypergraph
$S$, provided that $c=(1+o(1))|E(S)|$ and $p$ is (up to a
multiplicative constant) a threshold function for the property
of containment of a copy of $S$. In the second application we show that
a typical $G\sim \mathcal H_{c}^2(n,p)$ contains $(1-o(1))np/2$ edge
disjoint Hamilton cycles, each of which is rainbow, provided that
$c=\omega(n)$ and $p=\omega(\log n/n)$.

\end{abstract}

\section{Introduction}

We consider the following model of edge-colored random $k$-uniform
hypergraphs. Let $p\in [0,1]$ and let $c$ be a positive integer. Then we define
$\mathcal H^k_c(n,p)$ to be the probability space of edge-colored
$k$-uniform hypergraphs with vertex set $[n]:=\{1,\ldots,n\}$,
obtained by first choosing each $k$-tuple $e\in \binom{[n]}{k}$ to
be an edge independently with probability $p$ and then by coloring each
chosen edge independently and uniformly at random with a color from
the set $[c]$. For example, the case $k=2$ reduces to the standard binomial graph $\mathcal G(n,p)$, whose edges are colored at random in $c$ colors. In the special case where $c=1$, we write $\mathcal
H^k(n,p) := \mathcal H^k_c(n,p)$, and observe that this is just the
standard binomial random hypergraph model. For $H\sim \mathcal
H^k_c(n,p)$ and a hypergraph $S$, we say that $H$ contains a
\emph{rainbow} copy of $S$ if $H$ contains an isomorphic copy of
$S$ with all edges in distinct colors. A frequent theme in recent
research is to determine the conditions on $p$ and $c$ under which a
random hypergraph $H\sim \mathcal H^k_c(n,p)$ contains, with high
probability (w.h.p.), a rainbow copy of a given hypergraph $S$.

Let us first discuss the case $k=2$ of binomial random graphs.
Perhaps two of the most natural properties to address are when $S$
is a perfect matching or $S$ is a Hamilton cycle. Note that in these
cases we need $c\geq n/2$ and $c\geq n$, respectively. Moreover, it
is well known (see e.g., \cite{Bol}) that a perfect matching
(respectively a Hamilton cycle) starts to appear in a typical $G\sim
\mathcal G(n,p)$ whenever $p=\frac{\log n+\omega(1)}{n}$
(respectively, $p=\frac{\log n+\log\log n+\omega(1)}{n}$), and
therefore we can restrict ourselves to these regimes. In
\cite{cooper2002multi}, Cooper and Frieze showed that for $p\approx
42\log n /n$ and $c=21n$, a graph $G\sim \mathcal G_c(n,p)$
typically contains a rainbow Hamilton cycle. Later on, Frieze and
Loh~\cite{frieze2014rainbow} improved this to $p=\frac{(1+o(1))\log
n}{n}$ and $c =(1+o(1))n$, which is asymptotically optimal with
respect to both of the parameters $p$ and $c$. Recently, Bal and
Frieze \cite{bal2013rainbow} obtained the optimal $c$ by showing
that for $p=\omega(\log n/n)$ and $c = n/2$ (respectively $c=n$), a
graph $G\sim \mathcal G_c(n,p)$ w.h.p.\ contains a rainbow perfect
matching (respectively a rainbow Hamilton cycle). For general
graphs, Ferber, Nenadov and Peter \cite{ferber2013universality}
showed that for every graph $S$ on $n$ vertices with  maximum degree
$\Delta(S)$ and for $c=(1+o(1))e(S)$, a typical $G\sim \mathcal
G_c(n,p)$ contains a rainbow copy of $S$, provided that $p=
n^{-1/\Delta(S)} \text{polylog}(n)$ (here, as elsewhere, $e(S)$
denotes the number of edges in $S$). In this case, the number of
colors $c$ is asymptotically optimal, whereas the edge probability
$p$ is almost certainly not.

Our first main result improves the main theorem of Frieze and Loh
from \cite{frieze2014rainbow} to $p=\frac{\log n+\log\log
n+\omega(1)}n$, which is clearly optimal. Our proof technique is
completely different, resulting in a shorter proof than the one given in
\cite{frieze2014rainbow}.

\begin{theorem}
  \label{main}
    Let $\varepsilon>0$, let $c=(1+\varepsilon)n$ and let $p=\frac{\log n+\log\log
    n+\omega(1)}{n}$.  Then a graph $G\sim \mathcal G_c(n,p)$
    w.h.p.\ contains a rainbow Hamilton cycle.
\end{theorem}

Next, building upon an ingenious coupling idea of McDiarmid
\cite{McD}, we give a general statement regarding the problem of
finding ``nice" structures in randomly edge-colored random
hypergraphs. Then, we exhibit its applications to derive
interesting corollaries. Before doing so, let us introduce some
useful notation. For an integer $c$, suppose that $\mathcal
C:=\mathcal C(c,n,k)$ is a collection of edge-colored $k$-uniform
hypergraphs on the same vertex set $[n]$, whose edge set is colored
with colors from $[c]$. We say that $\mathcal C$ is
\emph{$\ell$-rich} if for any $C\in \mathcal C$ and for any $e\in
E(C)$ there are at least $\ell$ distinct ways to color $e$ in order
to obtain an element of $\mathcal C$. For example, consider the case
where $k=2$ and $c=n_1\geq n$, and let $\mathcal C(c,n,k)$ be the
set of all possible rainbow Hamilton cycles in $K_n$. Note that for
each $C\in \mathcal C$ and for every $e\in E(C)$, there are
$n_1-n+1$ ways to color $e$ in order to obtain a rainbow Hamilton
cycle. Therefore, $\mathcal C$ is $(n_1-n+1)$-rich. Now, given a
collection of edge-colored hypergraphs $\mathcal C$, we define
$\widetilde{\mathcal C}$ to be the set of all hypergraphs obtained
by taking $C\in \mathcal C$ and deleting the colors from its edges.
With this notation in hand we can state the following theorem:

\begin{theorem}
  \label{main2}
Let $p:=p(n)\in [0,1]$ and let $\ell, c$ be positive integers for
which $q:=\frac{cp}{\ell}\leq1$. Let $k$ be any positive integer and
let $\mathcal C:=\mathcal C(c,n,k)$ be any $\ell$-rich set. Then, we
have
$$\Pr\left[H\sim \mathcal H^k(n,p) \textrm{ contains some } C\in \widetilde{\mathcal C}\right]\leq
\Pr\left[H\sim H^k_c(n,q) \textrm{ contains some } C\in \mathcal C
\right].$$

\end{theorem}

Next we present two interesting applications
for Theorem \ref{main2}, combining it with known results. First, we show that one can find a rainbow
Hamilton cycles in a random hypergraph with an optimal (up to a
multiplicative constant) edge density when working with the
approximately optimal number of colors. Second, we present
an application which is somewhat different in nature. We show that one
can find ``many" edge-disjoint Hamilton cycles, each of which is
rainbow in a random graph. Before stating it formally, we need the
following definition. Let $H$ be a $k$-uniform hypergraph on $n$
vertices. For $0\leq \ell< k$ we define a Hamilton $\ell$-cycle as a
cyclic ordering of $V(H)$ for which the edges consist of $k$
consecutive vertices, and for each two consecutive edges $e_i$ and
$e_{i+1}$ we have $|e_i\cap e_{i+1}|=\ell$ (where we consider
$n+1=1$). It is easy to show that a Hamilton $\ell$-cycle contains
precisely $m_\ell:=\frac{n}{k-\ell}$ edges and therefore we cannot
expect to have one unless $n$ is divisible by $k-\ell$. (Note that we can
consider a perfect matching as a Hamilton $0$-cycle.)

The problem of finding the threshold for the existence of Hamilton
$\ell$-cycles in random hypergraphs has drawn a lot of attention
in the last decade. Among the many known results, it is worth
mentioning the one of Johansson, Kahn and Vu \cite{JKV}, who showed
that $p=\Theta(\log n/n^{k-1})$ is a threshold for the appearance of
a Hamilton $0$-cycle (that is, a perfect matching) in a typical
$H\sim \mathcal H^k(n,p)$, provided that $n$ is divisible by $k$. In
general, for every $\ell<k$, the threshold for the appearance of a
Hamilton $\ell$-cycle in a typical $H\sim \mathcal H^k(n,p)$
(assuming that $n$ is divisible by $k-\ell$) is around $p\approx
\frac{1}{n^{k-\ell}}$ (in some cases an extra $\log$ factor
appears). For more details we refer the reader to \cite{DF} and its
references.

Now we are ready to state our next result:

\begin{theorem}
  \label{app1}
  Let $0\leq \ell<k$ be two integers, and let $p\in [0,1]$ be such
  that
  $$\Pr\left[H\sim \mathcal H^k(n,p) \text{ contains a Hamilton
  }\ell\text{-cycle}\right]=1-o(1).$$
  Then, for every $\varepsilon>0$, letting $c=(1+\varepsilon)m_\ell$ and
  $q=\frac{cp}{\varepsilon m_\ell+1}$ we have
  $$\Pr\left[H\sim \mathcal H^k_c(n,q) \text{ contains a rainbow
  Hamilton }\ell\text{-cycle}\right]=1-o(1).$$
\end{theorem}

{\bf Remark:} Note that we allow to take $\varepsilon$ to be a
function of $n$ (or even $0$) in the statement above. Moreover, if we take $\varepsilon$ to
be a constant, then in particular we see that by losing a
multiplicative constant in the threshold,  a rainbow Hamilton
$\ell$-cycle w.h.p. exists. This for example reproves and extends
the first result obtained by Cooper and Frieze
\cite{cooper2002multi} in a very concise way.

The second application is regarding the problem of finding many
rainbow edge-disjoint Hamilton cycles in a typical $G\sim \mathcal
G(n,p)$. The analogous problem without the rainbow requirements is well
studied and quite recently, completing a long sequence of papers, Knox,  K\"uhn
and Osthus \cite{knox2013edge},
Krivelevich and Samotij \mbox{%DIFAUXCMD
\cite{krivelevich2012optimal}
}%DIFAUXCMD
and  K\"uhn and Osthus \cite{kuhn2014hamilton}
 solved this question for the entire
range of $p$. Combining these results with Theorem \ref{main2} we can
in particular obtain the following:

\begin{theorem}
  \label{rainbow packing}
 For every $0<\varepsilon<1$ there exists $C:=C(\varepsilon)>0$ such that for every $p\geq \omega(\log
 n/n)$ and $c=Cn$ the following holds:
 $$\Pr\left[G\sim \mathcal G_c(n,p) \text{contains }
 (1-\varepsilon)np/2 \text{ edge disjoint rainbow Hamilton cycles}\right]=1-o(1).$$
\end{theorem}

{\bf Notation.} Our graph theoretic notation is quite standard and
mainly follows that of \cite{West}. For $p\in [0,1]$ we let
$\mathcal H^k(n,p)$ denote the probability space of $k$-uniform
hypergraphs on vertex set $[n]$, obtained by adding each possible
$k$-subset of $[n]$ as an edge with probability $p$, independently
at random. In the special case $k=2$, we denote $\mathcal
G(n,p):=\mathcal H^2(n,p)$, the well studied binomial random graph
model. For an integer $c$, we let $\mathcal H^k_c(n,p)$ be the
probability space of edge-colored $k$-uniform hypergraphs on vertex
set $[n]$ obtained as follows. First, take $H\sim \mathcal
H^k(n,p)$, and then, to each edge, assign a color from $\mathcal
C:=[c]$ uniformly, independently at random. As before, in the case
$k=2$ we write $\mathcal G_c(n,p):=\mathcal H^2_c(n,p)$. Given a
subhypergraph $H'$ of an edge-colored hypergraph $H$, we say that
$H$ is \emph{rainbow} if all its edges receive distinct colors. For
a vertex $v\in V(H)$ we denote by $d_H^c(v)$ its \emph{color}
degree, that is, the number of distinct colors appearing on edges
incident to $v$. For an edge $e\in E(H)$, we let $c(e)$ denote its
color. Given a subset of vertices $W\subseteq V(H)$ and a subset of
colors $\mathcal C_0\subseteq \mathcal C$, we let $H[W;\mathcal
C_0]$ denote the subhypergraph of $H$ on a vertex set $W$ for which
$e\in \binom{W}{k}\cap E(H)$ is an edge of $H[W;\mathcal C_0]$ if
and only if $c(e)\in \mathcal C_0$. In case that $G$ is a graph,
given two disjoint subsets of vertices $S,W\subseteq V(G)$ and a
subset of colors $\mathcal C_0\subseteq \mathcal C$, we let
$G[S,W;\mathcal C_0]$ denote the bipartite subgraph of $G$ with
parts $S$ and $W$, and edges $sw\in E(G)$, where $s\in S$, $w\in W$
and $c(sw)\in \mathcal C_0$. Moreover, for two disjoint subsets
$S,W\subseteq V(G)$ and an integer $D$, we say that $G$ contains a
$D$-matching from $S$ to $W$ if there exists a rainbow subgraph $M$
of $G$ such that $d_M(s,W)=D$ for every $s\in S$ and $d_M(w)\leq 1$
for every $w\in W$.
%That is, a collection of vertex disjoint $D$-stars, each
%of which is centered at $S$, whose union forms a rainbow subgraph.

%Let $k\geq 2$. A loose Hamilton cycle in a $k$-uniform hypergraph is
%a cyclic ordering of its vertices such that the edges consist of
%consecutive vertices, and every two consecutive edges intersect on
%exactly one vertex.

We will frequently omit rounding signs for the sake of clarity of
presentation.

\section{Tools}\label{sec:tools}

In this section we introduce some tools and auxiliary results to be
used in our proofs.

\subsection{Probabilistic tools}

We will routinely employ bounds on large deviations of random
variables. We will mostly use the following well-known bound on the
lower and the upper tails of the binomial distribution due to
Chernoff (see \cite{AlonSpencer}, \cite{JLR}).

\begin{lemma}[Chernoff's inequality]
    \label{Chernoff}
    Let $X \sim \operatorname{Bin} (n, p)$ and let
    $\mu = \mathbb{E}(X)$. Then
    \begin{itemize}
        \item $\Pr[X < (1 - a)\mu] < e^{-a^2\mu/2}$ for every $a > 0$;
        \item $\Pr[X > (1 + a)\mu] < e^{-a^2\mu/3}$ for every $0 < a < 3/2$.
    \end{itemize}
\end{lemma}

We also make use of the following approximation for the lower tail
of a binomially distributed random variable.

\begin{lemma} \label{estimatingBin}Let $\frac{\log n}n\leq p\leq \frac{2\log n}{n}$, and let $0<\delta<1$ be such that $\left(\frac{e^2}{\delta}\right)^\delta e^{-1+\delta}\leq e^{-0.7}$. Then
$$\Pr[\emph{Bin}(n,p)\leq \delta np]\leq n^{-2/3}.$$
\end{lemma}

\begin{proof}
  Note that
  \begin{align*}\Pr[\Bin(n,p)\leq \delta np]&=\sum_{i=0}^{\delta np}\binom{n}{i}p^i(1-p)^{n-i}
  \leq \delta np \left(\frac{e}{\delta}\right)^{\delta np}e^{-(1-\delta)np}\\
  &\leq \delta np\left[\left(\frac{e}{\delta}\right)^{\delta}e^{-1+\delta}\right]^{np}
  \leq \delta np e^{-0.7np}\\
  &\leq e^{-(2/3)\log n}=n^{-2/3}.
    \end{align*}
\end{proof}

Before introducing the next tool to be used, we need the following
definition.

\begin{definition}
  Let $(A_i)_{i=1}^n$ be a collection of events in some probability
space. A graph $D$ on a vertex set $[n]$ is called a
\emph{dependency graph} for $(A_i)_i$ if $A_i$ is mutually
independent of all the events $\{A_j: ij\notin E(D)\}$.
\end{definition}

%
%The following is the so called Lov\'asz Local Lemma, in its symmetric version.
%
%\begin{lemma}(Local Lemma)\label{LLL}
%Let $(A_i)_{i=1}^n$ be a sequence of events in some probability
%space, and let $D$ be a dependency graph for $(A_i)_i$. Let $\Delta:=\Delta(D)$ and
%suppose that for every $i$ we have $\Pr\left[A_i\right]\leq q$ and that $eq(\Delta+1)<1$. Then, $\Pr[\bigcap_{i=1}^n \bar{A}_i]>0.$
%
%
%\end{lemma}

We make use of the following asymmetric version of the Lov\'asz
Local Lemma (see, e.g. \cite{AlonSpencer}).

\begin{lemma} (Asymmetric Local Lemma)\label{ALLL}
Let $(A_i)_{i=1}^n$ be a sequence of events in some probability
space. Suppose that $D$ is a dependency graph for $(A_i)_i$, and suppose that there are real numbers $(x_i)_{i=1}^n$
such that for every $i$ the following holds:
$$\Pr[A_i]\leq x_i\prod_{ij\in E(D)}(1-x_j).$$
Then, $\Pr[\bigcap_{i=1}^n \bar{A}_i]>0.$
\end{lemma}

\subsection{Properties of $\mathcal G_c(n,p)$}

Here we gather fairly standard typical properties of sparse binomial
random graphs. Given a graph $G=(V,E)$ with vertex set  $V=[n]$
vertices, define the set of vertices  $SMALL\subseteq V$ by
$$
SMALL:=\{v\in [n]: d_G(v)\leq \delta \log n\}\,,
$$
where $\delta>0$ is a small enough absolute constant.

\begin{lemma}\label{lemma:Properties} Let $0<\beta,\varepsilon<1$ be absolute constants, let $c=(1+\varepsilon)n$, and let $\frac{\log n}{n}\leq p\leq \frac{2\log n}{n}$. Then, w.h.p.\ a graph $G\sim \mathcal G_c(n,p)$
  satisfies the following properties.
  \begin{enumerate}[$(P1)$]
   \item $\Delta(G)\leq 10\log n$.
    \item $|SMALL|\le n^{0.4}$.
    \item For every $v\in [n]$, $d^c_G(v)\ge d_G(v)-2$.
    \item Let  $E_0=\{e\in E(G): e\cap SMALL\neq \emptyset\}$. Then all the
    elements of $E_0$ are of distinct colors.
    \item No two vertices $x,y\in SMALL$ ($x$ and $y$ might be the same) have a path of length at most $4$ with $x,y$ as its endpoints in $G$.
    \item For every two disjoint subsets $X$ and $Y$ of size
    $|X|=|Y|=\omega\left(\frac{n}{(\log n)^{1/2}}\right)$, the
    number of colors appearing on the edges between $X$ and $Y$ is at least $(1+\varepsilon-o(1))n$.
    \item For every subset $\mathcal C\subseteq [c]$ of size $\beta n$ and for every subset $X\subseteq [n]$ for which $|X|^2p=\omega(n)$ we have
    that $\frac{\beta}{3} |X|^2p\leq e(G[X;\mathcal C])\leq \beta |X|^2p$.
    \item For every subset $\mathcal C\subseteq [c]$ of size $\beta n$ and for every two disjoint subsets $X,Y\subseteq [n]$ such that $|X||Y|p=\omega(n)$, we have
    that $\frac{\beta}{2} |X||Y|p\leq e(G[X,Y;\mathcal C])\leq \beta|X||Y|p$.
    \item For every $s\in[c]$, the number of edges in $G$ which are colored in
    $s$ is at most $10\log n$.
    \item For every subset $X\subseteq [n]$, if $|X|\leq
\frac{n}{\log^{4 /3}n}$, then $e_G(X)\leq 8|X|$.
\item For every $X\subseteq [n]$ of size $
|X|\geq \frac{n}{\log^{4/3}n}$, we have $e_G(X)\leq
|X|^2p\left(\frac{n}{|X|}\right)^{1/2}$.

  \end{enumerate}
\end{lemma}

\begin{proof}

For $(P1)$, just note that given a vertex $v\in [n]$, since $d_G(v)\sim \Bin(n-1,p)$, it follows that

\begin{align*}
\Pr[d_G(v)\geq 10\log n]&\leq \binom{n}{10\log n}p^{10\log n}\\
&\leq \left(\frac{enp}{10\log n}\right)^{10\log n}=o(1/n).
\end{align*}

Therefore, by applying the union bound we obtain that w.h.p.\ $\Delta(G)\leq
10\log n$.

For $(P2)$ note that by Lemma \ref{estimatingBin}, the expected number of such vertices is at most $n^{1/3}$.
Therefore, by applying Markov's inequality, $(P2)$ immediately follows.

For $(P3)$,
%we first show that w.h.p.\ $d^c_G(v)=d_G(v)$ for every vertex $v$ with $d_G(v)=o(\log n)$. Indeed,
%given a vertex $v$, the probability that there exists a color from $[c]$
%which appears in at least two edges adjacent to $v$
%is upper bounded by $\binom{d_G(v)}{2}\frac{1}{c}=o(n^{-0.9})$. Now, since by $(P2)$ there are $o(n^{1/2})$ vertices of degree $o(\log n)$, by applying the union bound we get the desired.
%Now, suppose that $v\in [n]$ is such that $d_G(v)=\Omega(\log n)$. Suppose
aasume that $d_G^c(v)\leq d_G(v)-3$. In particular, this means that
there are 2 disjoint pairs $\{x_1,y_1\}$ and $\{x_2,y_2\}$ of
neighbors of $v$ such that for $i=1,2$ both $vx_i$ and $vy_i$ have
the same color. The probability of this to happen is upper bounded
by $\left(d_G(v)\right)^{4}c^{-2}=o(n^{-1})$. Therefore, by applying
the union bound we obtain $(P3)$.

For $(P4)$, note that by $(P2)$ we have that $|SMALL|=o(n^{0.49})$.
Now, since $\Delta(D)\leq 10\log n$, it follows that
$|E_0|=o(n^{1/2})$. Therefore,
$$\Pr\left[\exists e,e'\in E_0 \textrm{ with the same color}\right]\leq \binom{|E_0|}{2}c^{-1}=o(1).$$

For $(P5)$, note that, given two vertices $x,y$, the probability
that there exists a path of length at most $4$ between them is upper
bounded by $p+np^2+n^2p^3+n^3p^4\leq \frac{17\log^4n}n$. Now, since
by $(P2)$ we have $|SMALL|=o(n^{-0.49})$, and since the events
``$v\in SMALL$" are ``almost independent", by applying the union
bound we can easily obtain that the probability for having two such vertices
$x,y$ in $SMALL$ is $o(1)$.

For $(P6)$, let $X$ and $Y$ be two disjoint subsets. Note that for a
set $\mathcal C'\subseteq \mathcal C$ of size $t$, the probability
that none of the colors of $\mathcal C'$ appears on $E(X,Y)$ is upper
bounded by:
$$(1-p+p(1-t/c))^{|X||Y|}\leq e^{-pt|X||Y|/c}.$$
Therefore, if $t=\gamma n$ for some fixed constant $\gamma>0$, then
by applying the union bound we obtain that the probability for
having a subset $\mathcal C'$ of colors of size $\gamma n$, and two
disjoint subsets $X$ and $Y$ of sizes $|X|=|Y|=x=
\frac{10n}{(\gamma\log n)^{1/2}}$ for which none of the edges in
$E(X,Y)$ uses colors of $\mathcal C'$ is at most
$$
\binom{n}{x}\binom{n}{x}\binom{(1+\varepsilon) n}{\gamma
n}e^{-\gamma npx^2/c}\leq 8^ne^{-\frac{100\gamma pn^2}{c\gamma\log
n}} =o(1)\ .
$$

For $(P7)$ just note that the expectation of the number of such
edges is $\binom{|X|}{2}\cdot p\cdot \frac{\beta n}{c}=\omega(n)$, and therefore, by Chernoff's inequality and
the union bound over all choices of $X$ and of $\mathcal C$, we easily obtain the desired claim.

For $(P8)$, let $X,Y\subseteq [n]$ be such subsets. Note that since
$\mathcal C$ is of size $\beta n$, the expected number of edges
between $X$ and $Y$ which are assigned colors from $\mathcal C$
is $\frac{\beta}{1+\varepsilon}|X||Y|p=\omega(n)$. Therefore, the
property follows easily from Chernoff's inequality and the union
bound.

For $(P9)$, $s\in[c]$ be some color and let $Y$ denote the random
variable which counts the number of edges colored $s$ in $G$.
Clearly, $Y\sim \Bin\left(e(G),\frac{1}{(1+\varepsilon)n}\right)$.
Now, it is easy to show that w.h.p.\ $e(G)=
(1+o(1))\binom{n}{2}p\leq (1+o(1))n\log n\leq 2n\log n$, and
therefore,
\begin{align*}
\Pr\left[Y\geq 10\log n\right]&\le \binom{2n\log n}{10\log
n}\left(\frac{1}{(1+\varepsilon)n}\right)^{10\log n}\\
&\leq \left(\frac{2e n\log n}{10(1+\varepsilon)n\log
n}\right)^{10\log n}=o(1/n).
\end{align*}
Next, by applying the union bound we obtain the desired claim.

We leave $(P10)$ and $(P11)$ as an exercise for the reader.
\end{proof}

\subsection{Finding rainbow star matchings between appropriate sets}

In this subsection we describe the main technical lemma which will
be used in the proof of Theorem \ref{main}. Informally speaking, this
lemma ensures the existence of rainbow star matchings between sets
of appropriate sizes.

\begin{lemma}
  \label{lemma:ugly}
  Let $\alpha, \varepsilon>0$ be constants, let $D$ be a fixed integer, let $c=(1+\varepsilon)n$ and
  let $\frac{\log n}{n}\leq p\leq \frac{2\log n}n$. Then, a graph
  $G\sim \mathcal G_c(n,p)$ is w.h.p.\ such that the following
  holds. Suppose that
  \begin{enumerate}[$(i)$]
    \item $W\subseteq [n]$ of size $(1+o(1))\frac{n}{\log\log n}$, and
    \item $S\subseteq [n]$ of size $\frac{n}{\log^{0.4}n}\leq |S|\leq \frac{2n}{\log^{0.4}n}$, and
    \item $\mathcal C_1\subseteq \mathcal C:=[c]$ of size $|\mathcal
    C_1|=\alpha n$, and
    \item for every $s\in S$ there are at least $\frac{\log n}{(\log\log n)^2}$
    edges $e=sw$ with $w\in W$ and $c(e)\in \mathcal C_1$.
  \end{enumerate}
  Then, there exists a rainbow $D$-matching from $S$ to $W$, with all colors from $\mathcal C_1$.
\end{lemma}

The main ingredient in the proof of Lemma \ref{lemma:ugly} is the
following powerful tool due to Aharoni and Haxell \cite{AH},
generalizing Hall's theorem to hypergraphs.

\begin{theorem} \label{thm:AH}
Let $g$ and $D$ be positive integers and let $\mathcal H=\{\mathcal
H_1,\ldots, \mathcal H_t\}$ be a family of $g$-uniform hypergraphs
on the same vertex set. If, for every $I\subseteq [t]$, the
hypergraph $\bigcup_{i\in I}\mathcal H$ contains a matching of size
greater than $Dg(|I|-1)$, then there exists a function $f:[t]\times
[D]\rightarrow \bigcup_{i=1}^tE(\mathcal H_i)$ such that $f(i,j)\in
E(\mathcal H_i)$ for every $i$ and $j$, and $f(i,j)\cap
f(i',j')=\emptyset$ for $(i,j)\neq (i',j')$.
\end{theorem}

When applying Theorem \ref{thm:AH}, we will
distinguish between few cases according to the size of $I\subseteq
[t]$. The following lemmas will make our life a bit easier with it.

\begin{lemma}
  \label{lemma:auxForUgly1}
Let $\varepsilon>0$, let $c=(1+\varepsilon)n$, let $D\in \mathbb{N}$
and let $\frac{\log n}n\leq p\leq \frac{2\log n}n$. Then a graph
$G\sim \mathcal G_c(n,p)$ is w.h.p.\ such that the following holds.
For every collection of $j\leq \frac{n}{\log^{0.9}n}$ vertex
disjoint stars, each  of size $\log^{0.2}n$, the number of colors
appearing on their edges is at least $2Dj$.
\end{lemma}

\begin{proof}
Let $s:=\log^{0.2}n$. We show that the probability of having a
collection of $j\leq \frac{n}{\log^{0.9}n}$ stars, each of which of
size $s$ whose union contains at most $2Dj$ colors is $o(1)$. The
following expression is an upper bound for this probability:
\begin{align*}
\sum_{j=1}^{\frac{n}{\log^{0.9}n}}&\binom{n}{j}\binom{n}{s}^jp^{js}\binom{(1+\varepsilon)n}{2Dj}\left(\frac{2Dj}{(1+\varepsilon)n}\right)^{js}\\
&\leq \sum_{j=1}^{\frac{n}{\log^{0.9}n}}
\left(\frac{en}j\right)^j\left(\frac{e(1+\varepsilon)n}{2Dj}\right)^{2Dj}\left(\frac{2eDjnp}{s(1+\varepsilon)n}\right)^{js}=o(1).
\end{align*}

Indeed, fix $j\leq \frac{n}{\log^{0.9}n}$ and first choose $j$
vertices to be the "centers" of the stars. For each of these
vertices choose $s$ neighbors  and multiply by the probability of
all these edge to appear. Next, choose a set of $2Dj$ colors from
$c=(1+\varepsilon)n$ and multiply by the probability that all the
edges of the stars are  colored with these colors. This completes
the proof of the lemma.
\end{proof}

The following lemma may look at the first glance a bit complicated to
understand, but its role will become clear during the proof of Lemma
\ref{lemma:ugly}.

\begin{lemma}
  \label{lemma:auxForUgly2}
  Let $0<\alpha<\varepsilon<1$, let $c=(1+\varepsilon)n$, let $D\in \mathbb{N}$
and let $\frac{\log n}n\leq p\leq \frac{2\log n}n$. Then a graph
$G\sim \mathcal G_c(n,p)$ is w.h.p.\ such that for every
\begin{enumerate}[$(i)$]
\item $\frac{n}{\log^{0.9}n}\leq j\leq \frac{2n}{\log^{0.4}n}$,
\item $W\subseteq [n]$ of size $|W|=(1+o(1))\frac{n}{\log\log n}$, and
\item $\mathcal C_1\subseteq \mathcal C:=[c]$ of size $|\mathcal C_1|=\alpha
 n$,
\end{enumerate}
the following holds. The probability of having subsets $X\subseteq
[n]$ of size $j$, $W'\subseteq W$ and $\mathcal C_2\subseteq
\mathcal C_1$ of sizes at most $2Dj$ such that for every edge $xw\in
E_G(X,W)$ we have $c(xw)\in \mathcal C_2$ or $w\in W'$ or
$c(xw)\notin \mathcal C_1$ is $o(1)$.

\end{lemma}

\begin{proof}
In order to prove the lemma, note that we can upper bound the
probability by
\begin{align*}
\binom{n}{|W|}&\binom{|W|}{|W'|}\binom{c}{|\mathcal
C_1|}\binom{|\mathcal C_1|}{|\mathcal C_2|}\cdot
\sum_{j=\frac{n}{\log^{0.9}n}}^{\frac{n}{\log^{0.4}n}}\left(1-p+p\left(\frac{|\mathcal C_2|+|\mathcal C\setminus \mathcal C_1|}{|\mathcal C|}\right)\right)^{j|W\setminus W'|}\\
&\leq
16^n\sum_{j=\frac{n}{\log^{0.9}n}}^{\frac{n}{\log^{0.4}n}}\left(1-p+p\frac{2Dj}{(1+\varepsilon)n}+p(1-\alpha/(1+\varepsilon)+o(1))\right)^{j|W\setminus
W'|}\\
&\leq
16^n\sum_{j=\frac{n}{\log^{0.9}n}}^{\frac{n}{\log^{0.4}n}}\exp\left(-Cjp\left(\frac{n}{\log\log
n}-Dj\right)\right)=o(1).
\end{align*}

($C$ is some constant which depends on $\alpha$). This completes the
proof.
\end{proof}

The following lemma shows that in a typical random graph $G\sim
\mathcal G(n,p)$, any bipartite subgraph $B=(S\cup W,E')\subseteq G$
with all the vertices in $S$ of ``large" degree contains an
$s$-matching from $S$ to $W$, for an appropriate choice of
parameters.

\begin{lemma}
  \label{aux3}
Let $\frac{\log n}{n}\leq p\leq \frac{2\log n}n$. Then, a graph
$G\sim \mathcal G(n,p)$ is w.h.p.\ such that the following holds.
Suppose that $B=(S\cup W,E')$ is a bipartite (not necessarily
induced) subgraph of $G$, with
\begin{enumerate}[$(i)$]
\item $|S|\leq \frac{n}{\log^{0.9}n}$, and
\item $|W|=(1+o(1))\frac{n}{\log\log n}$, and
\item $d_B(v)\geq \frac{\log n}{(\log\log n)^2}$ for every $v\in S$.
\end{enumerate}
Then there exists a $\log^{0.2}n$-matching from $S$ to $W$.
\end{lemma}

\begin{proof} Let $B=(S\cup W,E')$ be the subgraph of $G$ as described
above. In order to prove the lemma, we use the following version of
Hall's Theorem (see, e.g., \cite{West}). A bipartite graph $B=(S\cup
W,E')$ contains an $s$-matching from $S$ to $W$, if and only if the
following holds:
\begin{equation}
  \label{eq:Hall}
  \text{ For every } X\subseteq S \text{ we have } |N_B(X)|\geq s|X|.
\end{equation}

Suppose that \eqref{eq:Hall} fails for $B$ with $s=\log^{0.2}n$.
Then, there exists a subset $X\subseteq S$  for which $|N_B(X)\cup
X|\leq (s+1)|X|$. In particular, letting $Y=N_B(X)\cup X$, by
$(iii)$, we conclude that $e_B(Y)\geq |Y|\frac{\log
n}{(s+1)(\log\log n)^2}\geq |Y|\frac{\log^{0.8}n}{2(\log\log n)^2}$.
Since $|Y|\leq (s+1)|X|\leq \frac{2n}{\log ^{0.7}n}$, and since
$|Y|\log^{0.7}n>|Y|^2p \left(n/|Y|\right)^{1/2}$ for every $|Y|\leq
n/\log^{0.6}n$, we obtain a contradiction to $(P10)$ and $(P11)$.
\end{proof}

Now we are ready to prove Lemma \ref{lemma:ugly}.

\begin{proof}[Proof of Lemma \ref{lemma:ugly}]
Let $\alpha<\varepsilon$, let $D\in \mathbb N$ and let $W,S\subseteq
[n]$ and $\mathcal C_1\subseteq \mathcal C$ as described in the
assumptions of the lemma. For every $s\in S$, we define a graph
$\mathcal H_s$ with vertex set $W\cup \mathcal C_1$ in the following
way. For every $w\in W$ and $x\in \mathcal C_1$, $wx\in E(\mathcal
H_s)$ if and only if $sw\in E(G)$ and $c(sw)=x$. Consider the family
$\mathcal H:=\{\mathcal H_s: s\in S\}$, and note that in order to
prove the lemma, we need to show that there exists a function
$f:S\times [D]\rightarrow \bigcup_{s\in S}E(\mathcal H_s)$ such that
$f(s,i)\in E(\mathcal H_s)$ for every $s$ and $i$, and $f(s,i)\cap
f(s',i')=\emptyset$ for $(s,i)\neq (s',i')$. To this end, we make
use of Theorem \ref{thm:AH}. All we need to show is that
for every $T\subseteq S$, the graph $\bigcup_{t\in T}E(\mathcal
H_t)$ contains a matching of size greater than $2D(|T|-1)$. We
distinguish between two cases:

{\bf Case 1:} $|T|\leq \frac{n}{\log^{0.9}n}$. Consider the
bipartite graph $B=(T\cup W,E')$, where $E':=\{tw: t\in T, w\in W
\text{ and } c(tw)\in \mathcal C_1\}$. By applying Lemma \ref{aux3}
to $B$, we conclude that there exists an $s$-matching from $T$ to
$W$ in $B$, where $s=\log^{0.2}n$. Let $M$ be such a matching, and
note that by applying Lemma \ref{lemma:auxForUgly1} to $M$, it
follows that the number of colors appearing in $M$ is at least
$2D|T|$. Now, one can easily deduce that $\bigcup_{t\in T}E(\mathcal
H_t)$ contains a matching of size at least $2D|T|>2D(|T|-1)$.

{\bf Case 2:} $\frac{n}{\log^{0.9}n}\leq |T|\leq |S|$. Let $M$ be a
matching in $\bigcup_{t\in T}E(\mathcal H_t)$ of maximum size, let
$\mathcal C_2:=\{x\in \mathcal C_1:\exists w\in W \text{ s.t } wx\in
M\}$ and let $W':=\{w\in W: \exists x\in \mathcal C_2\text{ s.t }
wx\in M\}$. Suppose that $|M|\leq 2D(|T|-1)<2D|T|$. In particular,
it means that for every $v\in T$ and $w\in W$ we have $vw\notin
E(G)$, or $c(vw)\in \mathcal C_2$, or $w\in W'$, or $c(vw)\notin
\mathcal C_1$, which contradicts Lemma \ref{lemma:auxForUgly2}. This
completes the proof.
\end{proof}

\subsection{Expansion properties of subgraphs of random graphs}

In the following lemma we show, that given an edge colored graph $G$,
one can find two subsets of colors $\mathcal C_1,\mathcal C_2$ and a
vertex subset $W$ which inherits some desired properties from $G$.
The statement of the lemma is adjusted so as to facilitate its
application in the proof of Theorem \ref{main}.

\begin{lemma}
  \label{lemma:vertex partition}
Let $0<\alpha,\delta,\varepsilon<1/2$ be constants and let $n$ be an
integer. Let $G$ be an edge colored graph on $m\geq (1-o(1))n$
vertices, and let $\mathcal C^*$ be its set of colors, of size
$|\mathcal C^*|\geq (1+\varepsilon/2)n$. Suppose that $\delta \log
n\leq \delta(G)\leq \Delta(G)\leq 10\log n$, that each color appears
at most $10\log n$ times in $G$, and that for each $v\in V(G)$,
$d^c_G(v)\geq d_G(v)-2$. Then one can find subsets $\mathcal C_0,
\mathcal C_1\subseteq \mathcal C^*$, and $W\subseteq V(G)$
satisfying  the following properties:
\begin{enumerate} [$(i)$]
\item $|W|=(1+o(1))\frac{n}{\log\log n}$, and
\item $\mathcal C_0$ and $\mathcal C_1$ are two disjoint subsets of sizes $(1+o(1))\alpha n$, and
\item for every $w\in W$, $d_{\mathcal C_0}(w,W)\in\left(\frac{\alpha d_G(w)}{2\log\log n},\frac{2\alpha d_G(w)}{\log\log
n}\right)$, and
\item for every $v\in V(G)$, $d_{\mathcal C_1}(v,W)\in
\left(\frac{\alpha d_G(v)}{2\log\log n},\frac{2\alpha
d_G(v)}{\log\log n}\right)$, and
\item for every $x\in \mathcal C_0$, $x$ appears on at most
$\frac{100\log n}{(\log\log n)^2}$ edges in $G[W]$.
\end{enumerate}
\end{lemma}

\begin{proof} Let $\mathcal C_0,\mathcal C_1\subseteq \mathcal C^*$ be two disjoint random subsets, obtained in the following way: for each element of $\mathcal C^*$ toss a coin with probability $2\alpha$ to decide whether it belongs to $\mathcal C_0\cup \mathcal C_1$, then, with probability $1/2$ decide to which of these sets it belongs.
All these choices are independent. Let $W\subseteq
V(G)$ be a random subset of vertices, obtained by picking each $v\in
V(G)$ with probability $\frac{1}{\log\log n}$, independently at
random. We wish to show that the obtained sets satisfy $(i)$-$(v)$
with positive probability. In order to do so, we consider several
types of events. First, let $A_W$ denote the event ``$|W|\notin
(1\pm o(1))\frac{n}{\log\log n}$". Second, for each $i\in \{0,1\}$,
let $C_i$ denote the event ``$|\mathcal C_i|\notin (1\pm o(1))\alpha
n"$. Third, for each vertex $v\in V(G)$, let $\Gamma_i(v)$ ($i\in
\{0,1\}$) denote the event ``$d_{\mathcal C_i}(v,W)\notin
\left(\frac{\alpha d_G(v)}{2\log\log n},\frac{2\alpha
d_G(v)}{\log\log n}\right)$". Lastly, for each $x\in \mathcal C^*$,
let $B_x$ be the event ``more than $\frac{100\log n}{(\log\log
n)^2}$ edges in $G[W]$ are colored $x$". With this notation at
hand, we wish to show that

$$\Pr\left[\overline{A_W}\cap \overline{C_1}\cap \overline C_2\cap
\left(\bigcap_{i\in \{0,1\},v\in
V(G)}\overline{\Gamma_i(v)}\right)\cap \left(\bigcap_{x\in \mathcal
C^*}\overline{B_x}\right)\right]>0.$$

First, define $\mathcal E:=\left\{A_W,C_1,C_2,\Gamma_i(v),B_x: v\in
V(G), i\in\{0,1\}, x\in \mathcal C^*\right\},$ and let us estimate
the probabilities of each of the events $X\in \mathcal E$. By using
Chernoff's bounds we trivially get
\begin{enumerate}[$(a)$]
\item $\Pr\left[A_W\right]= \exp\left(-\Theta(n/\log\log n)\right)$,
\item $\Pr\left[C_i\right]= \exp\left(-\Theta(\alpha n)\right)$, and
\item $\Pr\left[\Gamma_i(v)\right]= \exp\left(-C\alpha
d_G(v)/\log\log n\right)$, where $C$ is an absolute constant which
does not depend on $\alpha$.
\end{enumerate}

For estimating $\Pr\left[B_x\right]$, note that since $d^c_G(v)\geq
d_G(v)-2$ for every $v\in V(G)$, it follows that each color class
can be partitioned into at most four matchings, each of size at most
$10\log n$ (the maximum number of edges with the same color in $G$).
Fix such a partition (into matchings) for each color class $x$. It
follows that if $B_x$ fails, then in at least one of the four
matchings, at least $\frac{25\log n}{(\log\log n)^2}$ edges have
been chosen. Since in each matching these choices are independent,
and since for a fixed edge $e$, the probability that $e\in W$ is
$\frac{1}{(\log\log n)^2}$, it follows by Chernoff's bounds that

\begin{enumerate}[$(d)$]
\item $\Pr\left[B_x\right]\leq e^{-\frac{\log n}{2(\log\log n)^2}}$.
\end{enumerate}

Next, let us define a dependency graph $D$ for $\mathcal E$, where
the edges of $D$ are as follows:

\begin{itemize}
\item All pairs $A_WX$, where $X\in \mathcal E$, and
\item all pairs $C_iX$, where $i\in \{0,1\}$ and $X\in \mathcal E$,
and
\item all pairs $\Gamma_i(v)\Gamma_j(u)$, where $i\neq j$ and
$v=u$, or $v\neq u$ and $N_G(v)\cap N_G(u)\neq \emptyset$, or if the
same color $c\in \mathcal C^*$ appears on edges incident to both $u$
and $v$, and
\item all pairs $\Gamma_i(v)B_x$ for which there exists an edge $uw$
such that $\{u,w\}\cap \left(\{v\}\cup N_G(v)\right)\neq \emptyset$
and $c(uw)=x$, and
\item all pairs $B_xB_y$, for which there exist two edges $e$ and $f$,
of colors $x$ and $y$, respectively, such that $e\cap f\neq
\emptyset$.
\end{itemize}

Now, for some fixed constant $c_0> 1$, define $x_W=\sqrt{
\Pr\left[A_W\right]}$, $y_i=\sqrt{\Pr\left[C_i\right]}$ (where $i\in
\{0,1\}$), $x_{i,v}=c_0\Pr\left[\Gamma_i(v)\right]$ (where $i\in
\{0,1\}$ and $v\in V(G)$), and $z_x=\sqrt{\Pr[B_x]}$ for $x\in
\mathcal C^*$. Note that
$$\Pr\left[A_W\right]= \exp\left(-\Theta(n/\log\log n)\right)\leq x_W\prod_{i,v}(1-x_{i,v}),$$
and
$$\Pr\left[C_i\right]=\exp\left(-\Theta(\alpha n)\right)\leq
y_i\prod_{i,v}(1-x_{i,v}),$$ and
$$\Pr\left[\Gamma_{i}(v)\right]=\exp\left(-C\alpha
d_G(v)/\log\log n\right)\leq c_0x_{i,v}\prod_{j,u:
\Gamma_{j}(u)\Gamma_i(v)\in E(D)}(1-x_{j,u})\prod_{x\in N^c_G(v)}
(1-z_x),$$ and
$$\Pr\left[B_x\right]\leq z_x \prod_{x \in
N^c_G(v),i}(1-x_{i,v})\prod_{y: B_xB_y\in E(D)} (1-z_y).$$

(The last two inequalities hold because the corresponding ``degrees
of dependencies" are some $polylog (n)$). All in all, one can apply
the Asymmetric Local Lemma (Lemma \ref{ALLL}) and obtain the desired
claim.
\end{proof}

Now, let $\frac{\log n}n\leq p\leq \frac{2\log n}n$, let
$c=(1+\varepsilon)n$, and let $G\sim \mathcal G_c(n,p)$. We show
that w.h.p. $G$ is such that every (not necessarily induced)
subgraph $G_1\subseteq G$ on $(1+o(1))n/\log\log n$ vertices with
some degree constraints is also a very good expander.

\begin{lemma}
  \label{lemma:properties of G1}
Let $\alpha,\delta,\varepsilon>0$, let $\frac{\log n}n\leq p\leq
\frac{2\log n}n$, and let $c=(1+\varepsilon)n$. Then a graph $G\sim
\mathcal G_c(n,p)$ is w.h.p.\ such that the following properties
hold. Suppose that
\begin{enumerate} [$(i)$]

\item $W\subseteq [n]$ is of size $(1+o(1))n/\log\log n$, and
\item $\mathcal C_0\subseteq \mathcal C$ is of size $(1+o(1))\alpha
n$.
%\item no color of $\mathcal C_0$ appears in $G[W]$ more than
%$\frac{100\log n}{(\log\log n)^2}$ times.
%and
%\item $\frac{\alpha \delta \log n}{2\log\log n}\leq \delta(G[W;\mathcal
%C_0])\leq \Delta(G[W;\mathcal C_0])\leq \frac{20\alpha\log
%n}{\log\log n}$.
\end{enumerate}
Then $H:=G[W;\mathcal C_0]$ satisfies:
\begin{enumerate}[$(a)$]
\item For every subset $X\subseteq W$, if $|X|\leq
\frac{n}{\log^{4 /3}n}$, then $e_H(X)\leq 8|X|$, and
\item for every $X\subseteq W$ of size $\frac{n}{\log^{4/3}n}\leq
|X|\leq |W|$, we have $e_H(X)\leq
|X|^2p\left(\frac{n}{|X|}\right)^{1/2}$, and
\item for every $X\subseteq W$, if $|X|\geq n/\log^{0.4}n$, then
$e_H(X)\leq \alpha|X|^2p$, and
\item for every two subsets $X,Y\subseteq W$ satisfying
$|X||Y|p=\omega(n)$, we have $e_H(X,Y)\geq \frac{\alpha}{2}|X||Y|p$.
\end{enumerate}
\end{lemma}

\begin{proof}
  All these properties follow from the properties in Lemma
  \ref{lemma:Properties} and similar arguments, hence are left as an
  exercise for the reader.
\end{proof}

Let us define the following useful notion of a $(k,d)$-expander.

\begin{definition}
  \label{def:expander}
  A graph $G$ is called a \emph{$(k,d)$-expander} if for every
  subset $X\subseteq V(G)$ of size at most $k$ we have
  $$|N_G(X)\setminus X|\geq d|X|.$$
\end{definition}

The following lemma is almost identical to Lemma 2.4 in \cite{KLS}
(although with few minor modifications). For the convenience of the
reader, we briefly sketch the proof.

\begin{lemma}
  \label{RainbowExpander}
Let $0<\varepsilon,\delta<1$ and let $\alpha<\delta e^{-100}$ be
constants. Let $\frac{\log n}n\leq p\leq \frac{2\log n}n$, and let
$c=(1+\varepsilon)n$. Then there exists $d_0\in \mathbb{N}$ such
that for every $d\geq d_0$, a graph $G\sim \mathcal G_c(n,p)$ is
w.h.p.\ such that the following holds. Suppose that $W\subseteq [n]$
is of size $(1+o(1))n/\log\log n$, $\mathcal C_0\subseteq \mathcal
C$ of size $(1+o(1))\alpha n$ and $H:=G[W;\mathcal C_0]$ is a
subgraph of $G$ satisfying $\frac{\alpha \delta \log n}{2\log\log
n}\leq \delta(H)\leq \Delta(H)\leq \frac{20\alpha\log n}{\log\log
n}$ and
 properties (a)--(d) of Lemma \ref{lemma:properties of G1}. Moreover, assume that no color from $\mathcal C_0$ appears in $H$ more than $\frac{100\log n}{(\log\log n)^2}$ times. Then,
there exists a subgraph $R\subseteq H$ satisfying the following:
\begin{enumerate}[$(a)$]
\item $R$ is rainbow, and
\item $R$ is a $(k,100)$-expander (where $k=\alpha\delta |W|/100$), and
\item $|E(R)|\leq d|W|$.
\end{enumerate}
\end{lemma}

\begin{proof} (Lemma 2.4 in \cite{KLS}.) Let $d$ be a large enough
integer. Condition on $G$ satisfying all the
properties of Lemma \ref{lemma:Properties}. Now, for every $w\in W$,
let $w$ choose $d$ random edges of $H$ incident with $w$ (with
repetitions), and let $\Gamma(w)$ be the set of the edges chosen by
$w$. Let $R$ be the graph whose edge set is $\bigcup_{w\in
W}\Gamma(w)$. We wish to show that $R$ satisfies $(a)-(c)$ with
positive probability.

We consider few types of events. First, the events regarding the
rainbow part. For every two edges $e_1,e_2$ of the same color, let
us denote by $A(e_1,e_2)$ the event ``both $e_1$ and $e_2$ are in
$R$" (in case $e_1\neq e_2$), and ``$e_1$ is chosen in more than one
trial" (in case $e_1=e_2$). Define
$$\mathcal A:=\left\{A(e_1,e_2): e_1 \textrm{ and }e_2 \textrm{ have
the same color}\right\}.$$

Second, we consider the events ensuring the expansion of sets. For a
set $X\subseteq W$, let $B(X)$ denote the event that $e_R(X)\geq
\frac{d}{101}|X|$, and for every $\frac{n}{\log^{4/3}n}\leq x\leq
k$, let
$$\mathcal B_x=\{B(X):|X|=101x\}.$$

%The third type of events we consider are regarding the expansion of
%large sets. For two disjoint sets $X,Y$, let $C(X,Y)$ be the event
%$e_R(X,Y)=0$, and for every $k/d^2\leq x\leq k$, define
%$$\mathcal C_x=\{C(X,Y):X\cap
%Y=\emptyset,|X|=x,|Y|=|W|-3x\},$$ where for convenience, from now on
%we assume that $|W|=n/\log \log n$ (this will not affect any of the
%calculations bellow).

Clearly, if none of the events in $\mathcal A$ happens, then $(a)$
and $(c)$ hold. Now, assume in addition that none of the events
$\mathcal B_x$ happens, and we wish to show that $(b)$ holds. Let
$X\subseteq V(H)$ be a subset of size at most $k$, and we wish to
show that $|N_H(X)\setminus X|\geq 100|X|$. Assume otherwise, we
obtain a subset $X$ for which $|X\cup N_H(X)|<101|X|$. Since none of
the events in $\mathcal A$ holds, it follows that $|E_H(X\cup
N_H(X))|\geq d|X|> d|X\cup N_H(X)|/101$. Now, if $|X\cup N_H(X)|\leq
\frac{n}{\log^{4/3}n}$, then for a large enough $d$ it contradicts
$(a)$ of Lemma \ref{lemma:properties of G1}. If
$\frac{n}{\log^{4/3}n}\leq |X\cup N_H(X)|\leq 101k$, then it
contradicts the event  $B(X\cup N_H(X))$.

It thus suffices to show that with positive probability none of
these events occurs. To this end we make use of the Local Lemma. We
estimate the probabilities of each event above, define a dependency
graph $D$ and estimate its degrees.

{\bf The family $\mathcal A$:} For a fixed pair $e_1,e_2\in E(G)$
of the same color we have
$$\Pr\left[A(e_1,e_2)\right]\leq \left(\frac{4d \log\log
n}{\alpha\delta\log n}\right)^2.$$

Define $x=c_0\left(\frac{4d \log\log n}{\alpha\delta\log
n}\right)^2$ for some constant $c_0> 1$. For the ``degree of
dependency" within $\mathcal A$, note that $A(e_1,e_2)$ depends on
$A(f_1,f_2)$ if and only if $e_1\cup e_2$ intersects $f_1\cup f_2$.
Now, recall that $\Delta(H)\leq \frac{20\alpha \log n}{\log\log n}$,
and that each color class contains at most $\frac{100\log
n}{(\log\log n)^2}$ edges. Therefore, the number of events in
$\mathcal A$ which are neighbors of $A(e_1,e_2)$ in the dependency
graph is at most $4\Delta(H)\frac{100\log n}{(\log\log n)^2}\leq
\frac{8000\alpha \log^2n}{(\log\log n)^3}$.

{\bf The family $\mathcal B_t$:} For a fixed set $X$ of size $101t$,
similarly to \cite{KLS}, one can show that

$$\Pr\left[B(X)\right]\leq \binom{e_H(X)}{dt}\cdot (2d)^{dt}\cdot
\left(\frac{2\log\log n}{\alpha\delta \log n}\right)^{dt},$$
where for $t\leq n/(101\log^{0.4}n)$, by assumption $(b)$ of Lemma
\ref{lemma:properties of G1} we have $e_H(X)\leq 101^2t^2p
\left(\frac{n}{101t}\right)^{1/2}$, and therefore,

$$\Pr\left[B(X)\right]\leq \left(10^{9}\left(\frac{t}{n}\right)^{1/2}\log\log n/\alpha\delta\right)^{dt}.$$

For $t\leq n/(101\log^{0.4}n)$, define
$y_t=\left(10^9e\left(\frac{t}{n}\right)^{1/2}\log\log
n/\alpha\delta\right)^{dt/2}$ and note that for an appropriately large $d$
we have
$$\sum_{t=n/(\log^{4/3}n)}^{n/\log^{0.4}n} |\mathcal B_{t}|\cdot y_t=o(1).$$

Now, for $n/(101\log^{0.4}n)\leq t\leq k$, we use $(c)$ of Lemma
\ref{lemma:properties of G1} and obtain

$$\Pr\left[B(X)\right]\leq \left(\frac{10^9 t\log \log n}{\delta
n}\right)^{dt}.$$

Define $y_t=e^{C_1t}\left(\frac{4 et\log \log n}{\delta
n}\right)^{dt}$ for some constant $C_1>0$ and note that for
appropriate choices of $\delta,d$ and $C_1$ we obtain

$$\sum_{t=n/(101\log^{0.4}n)}^{k} |\mathcal B_t|\cdot y_t=\sum_{t=n/(101\log^{0.4}n)}^{\alpha\delta n/\log\log n}
\binom{n/\log\log n}{101t}e^{C_1t}\left(\frac{4 et \log\log
n}{\delta n}\right)^{dt}=o(1).$$

 To compute the ``degree of dependency" with members of $\mathcal
A$ in $D$, note that $B(X)$ is correlated with an event $A(e_1,e_2)$
if $(e_1\cup e_2)\cap X\neq \emptyset$. Therefore, since the maximum
degree of $H$ is at most $\frac{20\alpha \log n}{\log \log n}$ and
since each color appears at most $\frac{100\log n}{(\log\log n)^2}$
times, the number of events $A(e_1,e_2)$ correlated with $B(X)$ is
upper bounded by

$$|X|\frac{20\alpha \log n}{\log \log n}\frac{100\log n}{(\log\log n)^2}=|X|\frac{2000\alpha \log^2n}{(\log\log n)^3}.$$

In order to apply the Asymmetric Local Lemma (Lemma \ref{ALLL}) we need to show that the following
inequalities hold.

$$\Pr\left[A(e_1,e_2)\right]\leq x\cdot (1-x)^{\frac{8000\alpha
\log^2n}{(\log\log n)^3}}\left(\prod_t(1-y_t)^{|\mathcal
B_t|}\right),$$

$$\Pr\left[B(X)\right]\leq y_t\cdot (1-x)^{|X|\frac{2000\alpha \log^2n}{(\log\log
n)^3}}\cdot\left(\prod_t(1-y_t)^{|\mathcal B_t|}\right).$$

For the first inequality, note that since $x=c_0\left(\frac{4d
\log\log n}{\alpha\delta\log n}\right)^2$, it follows that

\begin{align*}
x\cdot (1-x)^{\frac{8000\alpha \log^2n}{(\log\log
n)^3}}\left(\prod_t(1-y_t)^{|\mathcal B_t|}\right)&\geq
c_0\left(\frac{4d \log\log n}{\alpha\delta\log
n}\right)^2e^{-2c_0\left(\frac{4d \log\log n}{\alpha\delta\log
n}\right)^2\frac{8000\alpha \log^2n}{(\log\log n)^3}}e^{-2\sum_t |\mathcal B_t|y_t}\\
&=(1+o(1))c_0 \left(\frac{4d \log\log n}{\alpha\delta\log
n}\right)^2\\
&\geq \Pr\left[A(e_1,e_2)\right].
\end{align*}

(we used the facts that $\sum |\mathcal B_t|\cdot y_t=o(1)$ and that
for small values of $x$ we have $1-x\geq e^{-2x}$).

The second inequality is even easier to verify and is left as
an exercise for the reader.
%{\bf The family $\mathcal C_t$:} Consider two sets $X$ and $Y$ of
%sizes $|X|=t$ and $|Y|=n/\log\log n-3t$, where $n/\log^{4/3}n\leq
%t\leq k$.
%
%Similarly to \cite{KLS}, we obtain
%
%$$\Pr\left[C(X,Y)\right]=\prod_{v\in
%X}\left(1-\frac{d_H(v,Y)}{\alpha \beta\log n/(2\log\log
%n)}\right)^d\leq e^{-\frac{de_H(X,Y)2\log \log n}{\alpha\beta \log
%n}}\leq e^{-\frac{2dt(n/\log\log n-3t)\log \log n}{\alpha\beta n}}
%,$$
%
%where here we used the fact that by $(P8)$ of Lemma
%\ref{lemma:Properties} we have $e_H(X,Y)\geq \alpha |X||Y|p/2$.
%Define $z_t=e^{C_2n/\log\log n}e^{-\frac{2dt(n/\log\log n-3t)\log
%\log n}{\alpha\beta n}}$, for some positive constant $C_2$ and note
%that
%
%$$\sum_{t=n/\log^{4/3}n}^{k}|\mathcal C_t|z_t\leq
%\sum \binom{\frac{n}{\log\log n}}{t}\binom{\frac{n}{\log\log
%n}}{3t}e^{C_2n/\log\log n}e^{-\frac{2dt(n/\log\log n-3t)\log \log
%n}{\alpha\beta n}}$$
\end{proof}

\subsection{Finding a long rainbow path}

In this section we state the following lemma, which follows almost
identically from the proof of Lemma 4.4 in \cite{BKS}. Before doing
so, we introduce the following definition:

\begin{definition}
  A graph $G$ on $n$ vertices whose set of edges is colored  is called
  \emph{$k$-rainbow-pseudorandom}, if for every two disjoint subsets
  of vertices $A,B\subseteq V(G)$ of size $|A|=|B|=k$, the number of colors appearing on
  the edges of $G$ between $A$ and $B$ is at least $n$.
\end{definition}

In the following lemma we show that a graph $G\sim \mathcal
G_c(n,p)$ is $k$-rainbow-pseudorandom in a ``robust" way.

\begin{lemma}\label{Gnp is k rainbow pseudo}
Let $0<\alpha<\varepsilon<1$ be two constants, let
$c=(1+\varepsilon)n$ and let $\frac{\log n}n\leq p\leq \frac{2\log
n}n$. Then a graph $G\sim \mathcal G_c(n,p)$ is w.h.p.\ such that
the following holds. For every subset $\mathcal C^*\subseteq [c]$ of
size $|\mathcal C^*|\geq (1+\alpha)n$, the graph $G[[n];\mathcal
C^*]$ is $k$-rainbow-pseudorandom for $k=\frac{n}{\log^{0.49}n}$.
\end{lemma}

\begin{proof}
  Follows immediately from Property $(P6)$ of Lemma
  \ref{lemma:Properties}.
\end{proof}

Now we state a modification of Lemma 4.4 from
\cite{BKS}.

\begin{lemma}\label{DFS}
Let $G$ be an edge-colored graph on $n$ vertices which is
$k$-pseudorandom. Then $G$ contains a rainbow path of length at
least $n-2k+1$.
\end{lemma}

\begin{proof}
  The proof is more or less identical to the proof of Lemma 4.4 in \cite{BKS}
  with some minor changes which are left to the reader.
\end{proof}

\subsection{Expander graphs}

Here we show that the union of few expander graphs yields an
expander graph as well.

First, we show that given an expander graph, by adding vertex
disjoint stars to it one cannot harm the expansion properties too
much.
\begin{lemma}\label{adding stars}
Let $G$ be a graph, let $m>1$, and let $k$ be a positive integer.
Let $S\subseteq V(G)$ be a subset of vertices for which there exists
an $m$-matching from $S$ to $V(G)\setminus S$. If $G[V(G)\setminus
S]$ is a $(k,m)$-expander, then $G$ is a $(k,(m-1)/2)$-expander.
\end{lemma}

\begin{proof} Let $X\subseteq V(G)$ of size $|X|\leq k$, and we wish
to show that $|N_G(X)\setminus X|\geq \frac{m-1}{2}|X|$. Let us
distinguish between the following two cases:

{\bf Case I:} $|X\cap S|\leq |X|/2$. In this case, since
$G[V(G)\setminus S]$ is a $(k,m)$, it follows that $X\setminus S$
expands by a factor of $c$ and therefore $|N_G(X)\setminus X|\geq
\frac{(m-1)|X|}{2}$.

{\bf Case II:} $|X\cap S|>|X|/2$. In this case, since there exists
an $m$-matching from $S$ to $V(G)\setminus S$, hence
 $|N_G(X)\setminus X|\ge |N_G(X\cap S)\cap (V(G)\setminus S)|-|X\setminus S|\ge m|X\cap S|-\frac{|X|}{2}\ge \frac{(m-1)|X|}{2}$.
\end{proof}

The following simple lemma is from \cite{BFHK} (Claim 2.8).

\begin{lemma} \label{BFHK}
Let $G$ be a graph, let $m>0$, and let $k$ be a positive integer.
Let $U\subseteq V(G)$ be a subset for which $d_G(u)\geq m-1$ for
every $u\in U$, and, moreover, there is no path of length at most
$4$ in $G$ whose (possibly identical) endpoints lie in $U$. If
$G[V(G)\setminus U]$ is a $(k,m)$-expander, then $G$ is a
$(k,m-1)$-expander.
\end{lemma}

\subsection{Boosters}

In the proof of Theorem \ref{main} we need to find a Hamilton path
between two designated vertices $x'$ and $y'$ in a sparse expander
subgraph $G_1$ of a typical $G\sim \mathcal G_c(n,p)$. Moreover, we
need such a Hamilton path to be rainbow within a prescribed subset of colors. In this section we show how
to achieve such a goal.

A routine way to turn a non-Hamiltonian expander graph $G_1$ into a
Hamiltonian graph is by using \emph{boosters}. A booster is a
non-edge $e$ of $G_1$ such that the addition of $e$ to $G_1$
decreases the number of connected components of $G_1$, or creates a
path which is longer than a longest path of $G_1$, or
 turns $G_1$ into a Hamiltonian graph. In order to turn $G_1$ into a
Hamiltonian graph, we start by adding a booster $e$ of $G_1$. If the
new graph $G_1\cup \{e\}$ is not Hamiltonian then one can continue
by adding a booster of the new graph. Note that after at most
$2|V(G_1)|$ successive steps the process must terminate and we end
up with a Hamiltonian graph. The main point using this method is
that it is well-known (for example, see \cite{Posa}) that a
non-Hamiltonian graph $G_1$ with ``good" expansion properties has
many boosters. However, our goal is a bit different. We wish to turn
$G_1$ into a graph that contains a rainbow Hamilton path with $x'$
and $y'$ as its endpoints. In order to do so, we add one (possibly)
fake edge $x'y'$ to $G_1$, color it with a new color (which does not
belong to $\mathcal C$) and try to find a rainbow Hamilton cycle
that contains the edge $x'y'$. Then, the path obtained by deleting
this edge from the Hamilton cycle will be the desired path. For that
we need to define the notion of \emph{$e$-boosters}.

Given a graph $G_1$ and a pair $e\in \binom{V(G_1)}{2}$, consider a
path $P$ of $G_1\cup \{e\}$ of maximal length which contains $e$ as
an edge. A non-edge $e'$ of $G_1$ is called an $e$-booster if
$G_1\cup\{e,e'\}$ has fewer connected components than $G_1\cup\{e\}$
has, or contains a path $P'$ which passes through $e$ and which is
longer than $P$, or that $G_1\cup\{e,e'\}$ contains a Hamilton cycle
that uses $e$. The following lemma is from \cite{FKN} and shows that
every connected and non-Hamiltonian graph $G_1$ with ``good"
expansion properties has many $e$-boosters for every possible $e$.

\begin{lemma} \label{lem:boosters}
Let $G_1$ be a connected graph for which $|N_{G_1}(X)\setminus X|\ge
2|X|+2$ holds for every subset $X\subseteq V(G_1)$ of size $|X|\le
k$. Then, for every pair $e\in \binom{V(G_1)}{2}$ such that
$G_1\cup\{e\}$ does not contain a Hamilton cycle which uses the edge
$e$, the number of $e$-boosters for $G_1$ is at least $(k+1)^2/2$.
\end{lemma}

\begin{remark}\label{rem:boosters}
The proof of Lemma \ref{lem:boosters} can be easily modified (in
fact, the same proof holds) for the following case. $G_1$ is a graph
obtained by adding not too many (say, at most $\text{\emph{polylog}}
n$) vertices of degree $2$, any two of them are far apart (say, of
distance at least $3$), to a $(k,3)$ expander, and $e$ is not
incident with any of these vertices.

As another remark, note that for a $(k,2)$-expander, we trivially
have that each connected component is of size larger than $k$,
and therefore, if the graph is not
connected, then there are at least $(k+1)^2$ boosters which decrease
the number of connected components.
\end{remark}

Note that in order to turn a rainbow graph $G_1$ into a graph that
contains a rainbow Hamiltonian cycle passing through $e$, one should
repeatedly add $e$-boosters, one by one, every time adding a
booster with an unused color, at most $2|V(G_1)|$ times. Therefore,
we wish to show that a graph $G\sim G_c(n,p)$ typically contains
``many" $e$-boosters of ``many" colors for every sparse expander
subgraph $G_1$ and every pair $e\in \binom{V(G_1)}{2}$.

\begin{lemma}\label{rainbow boosters}
Let $0<\varepsilon<1,\beta>0$, let $c=(1+\varepsilon)n$ and let
$\frac{\log n}n\leq p\leq \frac{2\log n}{n}$. Then a graph $G\sim
\mathcal G_c(n,p)$ is w.h.p.\ such that the following holds. Suppose
that
\begin{enumerate}[$(i)$]
\item $G_1\subseteq G$ is any subgraph with $\frac{n}{\log\log n}\leq |V(G_1)|\leq \frac{2n}{\log\log n}$ and
$|E(G_1)|=\Theta(n/\log\log n)$ which is an
$(\beta|V(G_1)|,2)$-expander, and
\item $e\in \binom{V(G_1)}{2}$ is any pair which is
not incident with vertices of degree $2$ in $G_1$, and
\item $\mathcal C_2\subseteq [c]$ is a subset of size at least
$\varepsilon n/100$,
\end{enumerate}
then, $G$ contains $e$-boosters for $G_1$ assigned with colors from
$\mathcal C_2$.
\end{lemma}

\begin{proof} Note first that by Remark \ref{rem:boosters} after Lemma
\ref{lem:boosters}, there are at least $\beta^2|V(G_1)|^2/2\geq
\frac{\beta^2 n^2}{2(\log\log n)^2}$ $e$-boosters for every such
$G_1$. Fix a subset $\mathcal C_2\subseteq [c]$ of size at least
$\varepsilon n/100$, and observe that the probability of $E(G)$ not
to contain any $e$-booster which is assigned with a color from
$\mathcal C_2$ is at most
$$
\left(1-p+p\frac{1+0.99\varepsilon
n}{(1+\varepsilon)n}\right)^{\frac{\beta^2n^2}{2(\log\log n)^2}}
\le(1-p)^{ \frac{\varepsilon\beta^2 n^2 }{300(\log\log n)^2}}\leq
\exp\left(-\frac{\varepsilon\beta^2 n\log n }{300(\log\log
n)^2}\right).$$

Now, taking the union bound over all subsets $V(G_1)\subseteq [n]$
of size $\frac{n}{\log\log n}\leq |V(G_1)|\leq \frac{2n}{\log\log
n}$ and over all subgraph $G_1$ of $G$ on vertex set $V(G_1)$ with
at most $\frac{Cn}{\log\log n}$ many edges (where $C$ is some fixed
constant), and over all subsets of colors $\mathcal C_2\subseteq
[c]$ of size at least $\varepsilon n/5$ we obtain that the
probability of having a counterexample is upper bounded by

\begin{align*}
&\sum_{t=n/\log\log n}^{2n/\log\log
n}2^c\binom{n}{t}\binom{\binom{t}{2}}{Ct}p^{Ct}\exp\left(-\frac{\varepsilon\beta^2
n\log n }{300(\log\log n)^2}\right)\\
&\leq \frac{2n}{\log\log
n}2^{(1+\varepsilon)n}2^n\left(\frac{enp}{C\log\log
n}\right)^{2Cn/\log\log n}\exp\left(-\frac{\varepsilon\beta^2 n\log
n }{300(\log\log n)^2}\right)\\
&\leq 8^n \exp\left(C\frac{2n}{\log\log n} \log \frac{enp}{C\log\log
n}\right)\exp\left(-\frac{\varepsilon\beta^2 n\log n }{300(\log\log
n)^2}\right)=o(1).
\end{align*}

 This completes the proof. \end{proof}

\section{Proof of Theorem \ref{main}}\label{sec:proof of main1}

In this section we prove Theorem \ref{main}.

\begin{proof} Let $G\sim \mathcal G_c(n,p)$, and let $\delta>0$ be a sufficiently small constant (to be specified later). Throughout the
proof we assume that $G$ satisfies all the properties
of the lemmas from the previous section.
% Lemma \ref{lemma:Properties}.

Our proof strategy goes as follows. For each vertex $v\in SMALL$ let
us arbitrarily choose a set $A(v)=\{x,y\}$ of exactly two distinct
neighbors of $v$ and set $V_0=SMALL \cup \left(\bigcup_{v\in
SMALL}A(v)\right)$, and $E_0=\{vz: v\in SMALL \textrm{ and }z\in
A(v)\}$. By $(P4)$ and $(P5)$ of Lemma \ref{lemma:Properties} we have that all
the $A(v)$'s are disjoint and that $E_0$ is rainbow. Let $\mathcal
C_{small}:=\{c(e): e\in E_0\}$ denote the set of colors used in
$E_0$, and let $\mathcal C^*:=\mathcal C\setminus \mathcal
C_{small}$ be its complement. Observe that by $(P2)$  of Lemma
\ref{lemma:Properties}  for a small enough $\delta$ we have
$|\mathcal C^*|\geq (1+\varepsilon/2)n$.

Now, note that by $(P5)$ of Lemma \ref{lemma:Properties} it easily
follows that $\delta(G[V\setminus V_0])\geq \delta \log n$.
Therefore, letting $\alpha=\min\{\varepsilon/5,\delta e^{-10}\}$, by
applying Lemma \ref{lemma:vertex partition} to $G[V\setminus V_0]$
we find subsets $W\subseteq [n]\setminus V_0$ and $\mathcal
C_0,\mathcal C_1\subset \mathcal C^*$ for which

\begin{enumerate}[$(i)$]
\item $|W|=(1+o(1))\frac{n}{\log\log n}$, and
\item $\mathcal C_0\cap \mathcal C_1=\emptyset$, and
\item $|\mathcal C_0|,|\mathcal C_1|=(1+o(1))\alpha n$, and
\item for every $v\in [n]\setminus SMALL$ we have $d_{\mathcal
C_1}(v,W) \in \left(\frac{\alpha\delta\log n}{2\log\log n},\frac{20
\alpha \log n}{\log\log n}\right)$, and
\item the subgraph $H:=G[W;\mathcal C_0]$ satisfies all the properties of Lemma \ref{lemma:properties of G1}.
\end{enumerate}

In order to find the desired rainbow Hamilton cycle we proceed as
follows. First, find a rainbow path $P$ of length
$n-n/\log^{0.4}n$ in $[n]\setminus (V_0\cup W)$ whose edges
receive colors from $\mathcal C^*\setminus (\mathcal C_0\cup
\mathcal C_1)$. The existence of such a path is ensured by Lemmas
\ref{Gnp is k rainbow pseudo} and \ref{DFS}. Second, let $x,y$
denote $P$'s endpoints, define $S=\left([n]\setminus (SMALL\cup
V(P)\cup W)\right)\cup \{x,y\}$ be the set of ``unused" vertices,
and consider the bipartite graph $B:=G[S, W;\mathcal C_1]$. Lemma
\ref{lemma:ugly} ensures that $B$ contains (say) a rainbow
$9$-matching $M$ from $S$ to $W$. Let $x'$ and $y'$ be two neighbors
(in $M$) of $x$ and $y$, respectively, and define $M':=M\setminus
\{e\in M: e\cap \{x,y\}\neq \emptyset\}$.

Next, by applying Lemma \ref{RainbowExpander} to $G[W;\mathcal
C_0]$, we find a subgraph $R\subseteq G[W;\mathcal C_0]$ which
satisfies the following:

\begin{enumerate}[$(a)$]
\item $R$ is rainbow, and
\item $R$ is an $(\alpha\delta|W|/100,100)$-expander, and
\item $|E(R)|=\Theta(n/\log\log n)$.
\end{enumerate}

Now, let us define $G_1$ to be the subgraph of $G$ on vertex set
$V_1:=[n]\setminus V(P)$, with edge set $M'\cup E_0\cup E(R)$.
Note that since $R$ is an $(\alpha\delta|W|/100,100)$-expander, and
since for $S':=S\setminus \{x,y\}$, there exists a $9$-matching from
$S'$ to $W$, it follows by Lemma \ref{adding stars} that adding $S'$
and $M'$ to $R$ yields an $(\alpha\delta|W|/100,4)$-expander. Now,
since by $(P5)$ we have that vertices in $SMALL$ are far apart, by
Lemma \ref{BFHK} it follows that $G_1$ is an
$(\alpha\delta|V_1|/200,2)$-expander with $\Theta(n/\log\log n)$
edges, and is clearly rainbow. Finally, we wish to turn $G_1$ (in
$G$) into a graph which contains a rainbow Hamilton path with $x'$
and $y'$ as its endpoints. Note that both $x'$ and $y'$ are not
neighbors of vertices of degree $2$ in $G_1$. Now, one can
repeatedly apply Lemma \ref{rainbow boosters} to $G_1$ with respect
to the set of available colors to obtain a rainbow Hamilton path of
$G_1$ connecting $x'$ to $y'$ which uses only colors not
appearing on $P$.(Each time we add a booster $e$ whose color
$c(e)\in \mathcal C^*\setminus (\mathcal C_0\cup \mathcal C_1)$
has not been used before we update  the set of available colors
by excluding $c(e)$. Since $|W|=o(n)$, along the process we
 still have a linear number of colors available, and thus Lemma \ref{rainbow boosters} applies..)
A moment's thought now reveals that such a path,
together with $P$ and the edges $xx'$ and $yy'$, yields a rainbow
Hamilton cycle in $G$. This completes the proof.
\end{proof}

\section{Proof of Theorem \ref{main2}} \label{sec:proof of main2}
In this section we prove Theorem \ref{main2}.

\begin{proof}
  Let us define the following sequence $\Gamma_0,\Gamma_1,\ldots,\Gamma_N$ of random edge-colored $k$-uniform hypergraphs, where $N=\binom{n}{k}$, in the following way: Let $e_1,\ldots,e_N$
be an arbitrary enumeration of all the elements of $\binom{[n]}{k}$.
Now, in $\Gamma_i$, for every $j> i$ we add the corresponding edge
with probability $p$, independently at random and assign it all the
colors in $[c]$ (these edges can be seen as multiple edges with
multiplicity $c$). For every $j\leq i$, we add $e_j$ to $\Gamma_i$
with probability $q$, independently at random and then assign it a
unique color from $[c]$ uniformly, independently at random. Note
that $\Gamma_0\sim \mathcal H^{k}(n,p)$ while $\Gamma_N\sim \mathcal
H^{k}_c(n,q)$. Therefore, in order to complete the proof it is
enough to show that
\begin{align*}
\Pr\left[\Gamma_i \textrm{ contains some } C\in \mathcal
C\right]\geq \Pr\left[\Gamma_{i-1} \textrm{ contains some } C\in
\mathcal C \right].
\end{align*}

To this end, expose all edges but $e_i$ and its color(s) in both spaces. There are three possible
scenarios:
\begin{enumerate}[$(a)$]
\item $\Gamma_{i-1}$ contains some $C\in \mathcal C$ not using $e_i$, or
\item $\Gamma_{i-1}$ does not contain any member $C\in \mathcal C$ even if we add $e_i$ with all the possible colors, or
\item $\Gamma_{i-1}$ contains a member of $\mathcal C$ if we add $e_i$ with all the possible colors.
\end{enumerate}

Note that in $(a)$ and $(b)$ there is nothing to prove. Therefore,
it is enough to consider case $(c)$. The crucial observation here is
that if $e_i$ is needed for finding a copy of some $C\in \mathcal
C$, then since $\mathcal C$ is $\ell$-rich, it follows that at least
$\ell$ colors are valid for $e_i$ in order to obtain such a copy.
Now, the probability for $\Gamma_{i-1}$ to contain a member of
$\mathcal C$ is precisely $p$ (recall that $e_i$ is crucial for this
aim and that we add $e_i$ with all possible colors), where the
probability for $\Gamma_{i}$ to have such a copy is at least $q
\ell/c= p$. This completes the proof.
\end{proof}

\section{Applications of Theorem \ref{main2}}\label{sec:app}

In this section we show how to use Theorem \ref{main2} in order to
derive Theorems \ref{app1} and \ref{rainbow packing}. For
Theorem \ref{app1} we prove a stronger statement from
which the proof immediately follows.

\begin{theorem}
  \label{rainbow spanning}
  Let $S$ be any $k$-uniform hypergraph on $n$ vertices with $m$ edges, and let $p$
  be such that $$\Pr\left[H\sim \mathcal H^k(n,p) \text{ contains a
  copy of }S\right]=1-o(1).$$ Then, for every $\varepsilon\geq 0$,
  letting $c=(1+\varepsilon)m$ and $q=\frac{cp}{\varepsilon m+1}$, if $q\leq 1$ then  we
  have
  $$\Pr\left[H\sim \mathcal H^k_c(n,q) \text{ contains a rainbow
  }S\right]=1-o(1).$$
\end{theorem}

\begin{proof} Let $\mathcal C$ be the set of all possible rainbow
copies of $S$ on $n$ vertices with colors from $[c]$,
where $c=(1+\varepsilon)m$. Note that for any $e\in E(C)$, $C-e$ has
exactly $m-1$ edges and since there are $(1+\varepsilon)m$ colors,
it follows that there are $\varepsilon m+1$ ways to color $e$ to
obtain a rainbow copy of $S$. All in all, $\mathcal C$ is
$(\varepsilon m+1)$-rich, and therefore by applying Theorem
\ref{main2} to $\mathcal C$ we obtain the desired claim.
\end{proof}

Now we prove Theorem \ref{rainbow packing} which informally speaking
states that for $c=\omega(n)$ and $p=\omega(\log n/n)$, in a typical
$G\sim \mathcal G_c(n,p)$ one can find $(1-o(1))np/2$ edge-disjoint
Hamilton cycles, each of which is rainbow.

\begin{proof}
  First, observe that for example by the main results  of
  \cite{knox2013edge},\cite{krivelevich2012optimal}, it follows in particular that for
  $p=\omega(\log n/n)$ we have
  $$\Pr\left[G\sim \mathcal G(n,p) \text{ contains }(1-o(1))np/2
  \text{ edge-disjoint Hamilton cycles}\right]=1-o(1).$$
  Now, let $C$ be such that $\frac{Cn}{(C-1)n+1}\leq
  1+\varepsilon/2$ and let $c=Cn$. Let us define $\mathcal C$ to be the family of
  all collections $C$ of
  $(1-\varepsilon/2)np/2$ edge-disjoint Hamilton cycles, each of
  which is rainbow.
  Note that for every $C\in \mathcal C$ and every $e\in E(C)$, since
  $e$ belongs to a given rainbow Hamilton cycle, there are at most $n-1$ colors
  which are forbidden for it. Therefore, there are
  $Cn-(n-1)=(C-1)n+1$ ways to color $e$ in order to obtain an
  element of $\mathcal C$ and we conclude that $\mathcal C$ is
  $((C-1)n+1)$-rich. Now, by applying Theorem \ref{main2} for $q=\frac{Cnp}{(C-1)n+1}$ we obtain
  $$\Pr\left[G\sim \mathcal G_c(n,q)\text{ contains }
  (1-o(1))np/2 \text{ edge-disjoint rainbow Hamilton
  cycles}\right]=1-o(1).$$
  All in all, since $q\leq (1+\varepsilon/2)p$ we obtain that
  $(1-o(1))np\geq
  \frac{(1-o(1))nq}{1+\varepsilon/2}\geq (1-\varepsilon)nq$
  as desired.
\end{proof}

\end{document}